\documentclass[10pt,a4paper]{amsart}
\usepackage{amssymb, amsmath, amsthm}

\usepackage{pstricks}
\usepackage{mathrsfs}
\usepackage{extarrows}
\usepackage{amsthm}
\usepackage{amssymb}
\usepackage{xypic,enumerate, comment, enumitem}
\usepackage[all,cmtip]{xy}
\usepackage{fancyvrb}
\usepackage{longtable}

\usepackage{hyperref}
\hypersetup{
    colorlinks=true,       
    linkcolor=blue,          
    citecolor=blue,        
    filecolor=blue,      
    urlcolor=blue           
}

\theoremstyle{plain}

\newtheorem{theorem}{Theorem}[section]
\newtheorem{proposition}[theorem]{Proposition}
\newtheorem{corollary}[theorem]{Corollary}
\newtheorem{lemma}[theorem]{Lemma}
\newtheorem{conjecture}[theorem]{Conjecture}
\newtheorem*{theorem*}{Theorem}
\newtheorem*{proposition*}{Proposition}
\newtheorem*{corollary*}{Corollary}
\newtheorem*{lemma*}{Lemma}
\newtheorem*{conjecture*}{Conjecture}

\theoremstyle{definition}
\newtheorem{definition}[theorem]{Definition}
\newtheorem{example}[theorem]{Example}
\newtheorem*{definition*}{Definition}
\newtheorem*{example*}{Example}

\theoremstyle{remark}
\newtheorem{remark}[theorem]{Remark}
\newtheorem*{remark*}{Remark}
\newtheorem{fact}[theorem]{Fact}

\newcommand{\bC}{\mathbb{C}}

\newcommand{\bP}{\mathbb{P}}

\newcommand{\cC}{\mathcal{C}}
\newcommand{\cO}{\mathcal{O}}

\newcommand{\cM}{\mathcal{M}}

\newcommand{\sS}{\mathscr{S}}

\newcommand{\Pic}{\operatorname{Pic}}

\newcommand{\BVA}{\operatorname{BVA}}

\newcommand{\Hilb}{\operatorname{Hilb}}

\renewcommand\bar\overline
\renewcommand\tilde\widetilde

\newcommand{\Sym}{\operatorname{Sym}}
\renewcommand{\geq}{\geqslant}
\renewcommand{\leq}{\leqslant}

\newcommand{\PP}{\bP}

\newcommand{\gon}{\operatorname{gon}}

\newcommand{\irr}{\operatorname{irr}}
\newcommand{\cg}{\operatorname{cov.gon}}
\newcommand{\cng}{\operatorname{conn.gon}}

\newcommand\blfootnote[1]{%
    \begingroup
    \renewcommand\thefootnote{}\footnote{#1}%
    \addtocounter{footnote}{-1}%
    \endgroup
}

\title[Irrationality of the Fano surface]{Measures of irrationality of the Fano surface of a cubic threefold}

\begin{document}

\author{Frank Gounelas}
\address{Mathematisches Institut \\
        Humboldt Universit\"at Berlin \\
        10099 Berlin\\
        Germany.}
\email{gounelas@mathematik.hu-berlin.de}

\author{Alexis Kouvidakis}
\address{Dept. of Mathematics and Applied Mathematics \\
        University of Crete \\
        70013 Heraklion \\
        Greece.}
\email{kouvid@uoc.gr}



\begin{abstract} 
For $X$ a smooth cubic threefold we study the Pl\"ucker embedding of the Fano surface of lines $S$ of $X$. We prove that
if $X$ is general then the minimal gonality of a covering family of curves of $S$ is four and that this happens for a
unique family of curves. The analysis also shows that there is a unique pentagonal connecting family of curves, which
leads to the fact that the connecting gonality of $S$ is five whereas the degree of irrationality, i.e.\ the minimal
degree of a rational map from $S$ to $\PP^2$, is six.
\end{abstract}

\maketitle

\section{Introduction}

\blfootnote{\today}

Establishing whether a given smooth projective variety is rational, namely birational to projective space, is a
classical and hard problem. Advances in resolution of singularities and cohomological techniques have led to spectacular
advances in such questions over the past fifty years, but as we lack a numerical criterion for rationality in higher
dimensions, results are obtained very much on a case by case basis. On the other hand recent developments in the
decomposition of the diagonal \cite{voisinbook}, have led to further clarity on the behaviour of rationality in families
of smooth varieties \cite{htp}. Given though that rational varieties are only a small part of the overall geography of
varieties, it is natural to attach invariants to any projective variety which measure how far it is either from being
rational or at least covered by rational curves. It turns out there are some very natural candidates for such invariants
and their behaviour follows the same patterns as rationality, uniruledness and rational connectedness respectively. We
begin by recalling the definitions of these invariants, the ideas going back to \cite{lopezpirola},
\cite{yoshihara} (see \cite{bdelu} for further references).

We work over the complex numbers $\bC$ throughout the paper. Let $C$ be a (possibly singular) integral curve. We define the
\textit{gonality} $\gon(C)$ to be the least integer $d$ so that there exists a degree $d$ morphism from the normalisation of $C$
to $\bP^1$. In this paper we are concerned with the following invariants.

\begin{definition}
    The \textit{covering gonality} $\cg(S)$ of an irreducible projective variety $S$ is the least integer $d$ such that
    through a general point of $S$ there passes an integral $d$-gonal curve. Similarly, the minimal $d$ so that through
    two general points of $S$ there passes an integral $d$-gonal curve is called the \textit{connecting gonality} $\cng(S)$.
\end{definition}

We point out that $\cg(S)=1$ (resp.\ $\cng(S)=1$) if and only if $S$ is uniruled (resp.\ rationally connected). In other
words, these two invariants measure how far a variety is from being uniruled or rationally connected. Analogously, the
following invariant measures how far a variety is from being rational.

\begin{definition}
    Let $S$ be an irreducible projective variety. One defines the \textit{degree of irrationality} of $S$ as the integer
    $$\irr(S) := \min\{n : \text{there is a dominant map }S\dashrightarrow\bP^{\dim S}\text{ of degree $n$}\}.$$
\end{definition}

It is easy to see (by pulling back lines from $\PP^{\dim S}$) that $$\cg(S)\leq\cng(S)\leq\irr(S).$$ The problem of
determining the above invariants is thus closely related to the existence of special families of curves on $S$. The key
point of the following definition is that the base of the family is not assumed projective and that every fibre maps
birationally onto its image.

\begin{definition}
    Let $S$ be an irreducible projective variety and $\cC\to T$ a smooth projective morphism over a quasi-projective
    irreducible base variety $T$, with irreducible fibres of dimension one. We say that $\cC/T$ is a \textit{covering
    family} of $S$, if there exists a dominant morphism $F:\cC\to S$ so that the restriction $F_t:\cC_t\to S$ to the
    general fibre is birational. We say that $\cC\to T$ is a \textit{connecting family} if in addition the induced
    $\cC\times_T\cC \to S\times S$ is dominant. Furthermore, we say that $\cC/T$ is a \textit{covering family of
    $d$-gonal curves} (resp.\ \textit{connecting family of $d$-gonal curves}) if the general fibre $\cC_t$ has gonality
    $d$.
\end{definition}

In Section \ref{defsection} we study the behaviour of the above invariants in families, in particular proving that $\cg$
does not increase under specialisation (a property also satisfied by gonality in a family of curves) and that it is lower
semicontinuous. We note that these results extend the corresponding results on uniruled varieties of \cite[IV
1.8.2]{kollar}. The techniques used are similar in nature to loc.\ cit.\ and draw from the theory of relative Hilbert
schemes, relative Kontsevich moduli spaces and admissible covers. We note that the behaviours of connecting gonality and
degree of irrationality in families are more subtle. For example $\irr$ is not lower semicontinuous, since as seen in
\cite{htp} there exists a smooth family of varieties where the very general fibre is not rational (in particular
$\irr>1$) yet countably infinitely many fibres are rational (cf.\ \cite{dff}). On the other hand it has been proven
recently \cite{ns}, \cite{kt} that a (mildly singular) specialisation of smooth rational varieties must be rational. It
would be interesting to know whether these results imply that $\irr$ can only drop under specialisation in a smooth
family. See Remark \ref{remarkconngon} for some observations on the behaviour of $\cng$.

A flurry of recent work computes and bounds these invariants for various types of varieties, for example hypersurfaces
\cite{bdelu}, \cite{bcfs} and symmetric products of curves \cite{bastianelli} (see other references therein). The aim of
this paper is to study these invariants for the Fano scheme of lines of a cubic hypersurface. Applications of this analysis
are twofold. Firstly, we can compute explicitly the above invariants in many new interesting cases by exploiting
Grassmannian geometry. Secondly, by obtaining for example the covering gonality of the Fano scheme of lines one deduces
information about the various surfaces in the hypersurface itself which are ruled by lines.

We begin with some background on Fano schemes of lines in the first case of interest, namely for cubic hypersurfaces.
Let $n\geq 3$ and $X\subset\bP^{n+1}_\bC$ a smooth cubic hypersurface and $F(X)\subset \rm{G}(2,n+2)$ the Fano scheme
parametrising lines contained in $X$. From \cite[1.3, 1.12, 1.16]{ak} we know that $\dim F(X)=2(n-2)$ and $F(X)$ is a
smooth irreducible variety. In \cite[1.8]{ak} it is proven that the canonical line bundle $K_{F(X)}$ is isomorphic to
$\cO_{F(X)}(4-n)$ where $\cO_{F(X)}(1)$ is the very ample line bundle coming from the Pl\"ucker embedding $F(X)\subset
\rm{G}(2,n+2)$. In particular, if $n\geq5$ then $F(X)$ is Fano and hence rationally connected. This implies
$\cg(F(X))=\cng(F(X))=1$.

If $n=4$ then from \cite{bd} the Fano scheme of lines $F(X)$ is a smooth hyperk\"ahler fourfold and as Voisin points out
in her notes \cite[Example 2.18]{voisinnotes}, we know that $X$ admits a two dimensional family of hyperplane sections,
every member of which is a cubic threefold with exactly three isolated nodes. The Fano surface of each of
these nodal threefolds sits inside $F(X)$. We know though that such a Fano surface is irreducible and is dominated by the
symmetric product of a $(2,3)$ curve in $\bP^3$ with two nodes, sitting inside a quadric cone. In particular these
surfaces are dominated by the self-product of a smooth genus two (hence hyperelliptic) curve, and so have covering
gonality two. Since these Fano surfaces move in a two dimensional family, they cover $F(X)$, which hence also has
covering gonality two as it is not uniruled.

In this paper we begin by studying the remaining case of the covering gonality of Fano schemes of lines of cubic
hypersurfaces.

\begin{theorem}\label{covgongen}
    Let $X$ be a general cubic threefold and $S$ its Fano surface. Then $$\cg(S)=4.$$ For any smooth cubic threefold, we
    have $$3\leq\cg(S)\leq4.$$
\end{theorem}

Since $K_S$ is very ample it is not hard to see also that for any smooth cubic threefold, the covering gonality
of $S$ is always at least three (see Lemma \ref{bva1}), so the point of the theorem is to show that four is
achieved for any smooth cubic threefold $X$ and that three does not happen for the general one. Next, we prove that
covering gonality four is in fact achieved by an (essentially) unique explicit family of curves and the same for
connecting gonality five, which we now describe. Given a line $\ell\subset X$, one defines the \textit{incidence
divisor} $D_\ell\subset S$ parametrising lines meeting $\ell$. We recall from \cite[Definition 6.6]{cg} (see also
Section \ref{fanosurface}) that a line $\ell$ is either of \textit{first} or of \textit{second type}. The locus of lines
of second type for a general $X$ is a smooth irreducible curve in $S$. In the following and the remainder of the paper,
the word unique is meant to signify that the family is considered up to closed proper subset of the base.


\begin{proposition}\label{covgonfour}
    For $X$ a general cubic threefold and $S$ its Fano surface, the family of incidence divisors $D_\ell$ of lines
    $\ell$ of second type is the unique tetragonal covering family of curves. Similarly, the family of
    incidence divisors of lines of first type is the unique pentagonal connecting family of curves.
\end{proposition}

As a consequence we can compute the gonality of incidence divisors.

\begin{corollary}\label{corincidencediv}
    Let $X$ be a general cubic threefold and $S$ its Fano surface. If $\ell$ is a general line of first type in $X$ then
    the incidence divisor $D_\ell\subset S$ is a smooth curve of gonality five. If $\ell$ is general of second type then
    $D_\ell$ is smooth of gonality four.
\end{corollary}

Another application of the analysis involved in Proposition \ref{covgonfour} concerns $\irr(S)$ and $\cng(S)$ for $X$
general.

\begin{theorem}\label{degofirr}
    Let $X$ be a general cubic threefold and $S$ its Fano surface. Then 
    \begin{eqnarray*}
        \cng(S)=5 \text{ and } \irr(S)=6.
    \end{eqnarray*}
    For any smooth cubic threefold we have $$3\leq\cng(S)\leq\irr(S) \leq 6.$$
\end{theorem}

We do not know any examples of smooth cubic threefolds where $\cg(S)=3$, $\irr(S)<6$ or $\cng(S)<5$ occurs. 

The paper is organised as follows. In section two we prove an invariance property of covering gonality in families. In
the third section we recall definitions and basic properties of the Fano surface of a cubic threefold and extend various
classical statements to cases we will need in later sections. In the fourth section we set up different notions of
separation of points by hyperplanes, which will ultimately lead to ruling out low gonality covering families of curves
in later sections. The fifth section is concerned with technical propositions listing configurations of three, four and
five lines respectively which cannot be separated by 2-planes in $\PP^4$. This is independent of previous sections but
lies at the heart of the proof of the main theorems. The penultimate section contains a description of the Klein cubic
(which we will be degenerating to in the main proof) and a resolution of its Hessian hypersurface whereas the final
section contains the proofs of the main theorems. 

The strategy of the proof of the above theorems is to use sections of the canonical bundle $K_S$ to separate points on
$S$ in the Pl\"ucker embedding. When such separation is possible, one obtains conclusions on the gonality of the general
member of a covering family of curves. We then use a blend of Grassmannian geometry and technical facts about the Fano
scheme of lines of a cubic threefold to rule out configurations of points on $S$ which cannot be separated by sections
of $K_S$. 

We emphasise that in Theorem \ref{degofirr} the computation of the connecting gonality follows from knowing that there
is a unique explicit covering family of curves, and likewise that the computation of the degree of irrationality follows
from this and also the fact there is a unique explicit connecting family of curves. An outcome of this study, in other
words, is that the detailed analysis of the families calculating each invariant helps compute the next one.

\subsection*{Acknowledgements}

This work began during a sabbatical visit of the second author at Humboldt University, whose hospitality is greatly
acknowledged. We also thank D. Agostini, G. Farkas, H-Y. Lin, J. Ottem for helpful conversations and in particular A.
Verra for showing us the geometric construction of how incidence divisors of lines of first type have a $g^1_5$. We are
also indebted to the anonymous referee for their very careful reading and in particular for pointing
out inaccuracies and many simplifications in the final four sections, which ultimately led to the strengthening of our
results.

\section{Some basic properties of covering gonality}\label{defsection}

In a family of curves it is clear that the gonality of the fibres is not necessarily constant, for example a family of
trigonal curves degenerating to a smooth hyperelliptic curve, or a family of elliptic curves degenerating to a singular
rational one. From semicontinuity it follows that gonality can only drop under specialisation, since the number of
sections of a line bundle can only increase. We will see that the same is true for covering gonality. On the other hand,
see the introduction for some remarks about the behaviour of degree of irrationality in families. 

The following is a standard application of general theory of Hilbert schemes and loci of $k$-gonal curves in families.
It is also mentioned with omitted proof in \cite[Section 1]{bdelu}.

\begin{lemma}\label{dgonalcovfamily}
    Let $S$ be an irreducible projective variety. Then $S$ has covering gonality (resp.\ connecting gonality) $d$ if and
    only if $d$ is the minimal positive integer so that there exists a covering family (resp.\ connecting family) $\cC\to
    T$ of $d$-gonal curves. 
\end{lemma}
\begin{proof}
    One direction is clear. For the other, we prove the case of covering gonality as that of connecting gonality is
    analogous. Assume $\cg(S)=d$ and let $\mathscr{H}$ be the Hilbert scheme of all one dimensional subschemes of $S$
    (note this is not of finite type). Since $S$ is over an uncountable algebraically closed field and the Hilbert
    scheme has countably many irreducible components, there has to be one irreducible component whose universal family
    contains enough $d$-gonal integral curves to dominate $S$. Call this irreducible component $T$ and consider the
    universal family over the function field of this component. By resolving singularities we can assume there is a
    quasi-projective variety $T'$ (mapping finitely to an open inside $T$) and a family of smooth connected curves which
    are the normalisations of the corresponding dimension one schemes of the universal family over $T$. Now, in such a
    family it is known that the locus of curves admitting a $g^1_d$ is possibly not connected, but is at least a
    variety, since it is the image of the pullback of the space $\mathcal{G}^1_d$ from $\mathcal{M}_g$ (see
    \cite{acgh}). Hence the locus of curves admitting a $g^1_d$ has finitely many irreducible components. Since the
    general point of $S$ has a $d$-gonal curve passing through it, one of these components has to be large enough so
    that the universal family over it dominates $S$.
\end{proof}

From the above lemma we will interchangeably use the original definition of covering gonality, or that of the minimal
integer $d$ so that there exists a covering family of $d$-gonal curves. It is often useful to reduce the above constructed
family to one where $\cC,T$ are both smooth and $\cC\to S$ is generically finite (see \cite[1.5]{bdelu}). Even
though we will not strictly need it for this paper (we will instead degenerate other notions than covering gonality),
the following is related to the previous lemma and we include its proof for posterity.

\begin{proposition}\label{propgonalitydrops}
    Let $f:\mathscr{S}\to U$ a flat family of irreducible projective varieties over a pointed irreducible one
    dimensional variety $(U,0)$, and assume that for $u\neq0$ the fibre of $f$ has covering gonality $d$. Then the
    special fibre $\sS_0:=f^{-1}(0)$ has covering gonality at most $d$. Moreover, covering gonality is lower
    semicontinuous.
\end{proposition}
\begin{proof}
    The argument is similar to the previous lemma and essentially reduces to families of curves. Let
    $\mathscr{S}^0:=\mathscr{S}\setminus f^{-1}(0)$, $U^0 := U \setminus\{0\}$ and $f^0:\mathscr{S}^0 \to U^0$. Consider
    the relative Hilbert scheme $\mathscr{H} := \Hilb_1(\sS^0/U^0)\to U^0$ parametrising one dimensional subschemes
    inside fibres of $f^0$ and the subset $\mathscr{G} \subset \mathscr{H}$ of integral curves which are $d$-gonal. Since
    this last space has countably many irreducible components and a fibre of $f^0$ has covering gonality $d$, we
    can single out as in the previous lemma one component whose universal family dominates every fibre of $f^0$, and after
    shrinking and blowing up as before we obtain a family $W\to U^0$ of subvarieties of the relative Hilbert scheme
    $\mathscr{H}$, such that the universal family over it is a family of curves $\cC\to W$ having the following
    property: for any $u\in U^0$ the family $\cC_u\to W_u$ is a smooth family of $d$-gonal curves covering
    $\mathscr{S}_u$. In particular, by passing to the corresponding irreducible component of the compactified relative
    Kontsevich moduli space $\bar{\cM}_g(\sS/U)$ (see \cite[50]{kollararaujo} for the relative version of the
    theorems of \cite{fp}) we complete to a family of stable morphisms $\cC\to W$ over $U$. There is an induced morphism
    $F:\cC\to \mathscr{S}$ so that $F:\cC_u\to\mathscr{S}_u$ is dominant for all $u\in U^0$, with image
    $\mathcal{Z}:=F(\cC)$ fitting in the diagram
    $$\xymatrix{
        \cC\ar[r]^F\ar[d] & \mathcal{Z}\ar[d]^\pi \\
        W\ar[r] & U
    }$$
    so that $\mathcal{Z}_u\subset\mathscr{S}_u$ for all $u\in U$ and $\dim\mathcal{Z}_u = \dim \mathscr{S}_u$ for all
    $u\in U^0$, which we now want to prove for $u=0$ also. Since $\pi$ is proper and fibre dimension is upper
    semicontinuous we obtain also that $\dim\mathcal{Z}_0 = \dim \mathscr{S}_0$, namely that the family of curves
    $\cC_0\to W_0$ is a covering family of curves of $\mathscr{S}_0$ (cf.\ \cite[Proof of IV.1.8.1]{kollar}). It could
    though possibly be that the special fibre over $0\in U$, namely $\cC_0\to W_0$, consists generically of reducible
    nodal curves. We will show that the resolution of each irreducible component of the general fibre of $\cC_0\to W_0$
    has a $g^1_k$ for $k\leq d$. 
    
    By mapping the base of this family via the stabilisation map to $\bar{\cM}_g$, the image is in the closure of the
    $d$-gonal locus, namely the image of the compactified Hurwitz space $\bar{\mathcal{H}}_{d,g}\to\bar{\cM}_g$ of
    admissible covers (see \cite[p175-186]{harrismorrison}). Therefore the stable model of every member of the
    central fibre over $U$ is also the stable model of such an admissible cover. By the definition of admissible covers,
    every irreducible component of every member of $\cC_0\to W_0$ has a $g^1_k$ for $k\leq d$ as required. 
\end{proof}

\begin{remark}\label{remarkconngon}
    We note that the above statement is not true for connecting gonality since we are assuming the connecting curves are
    integral. In the above it could be that the induced family $\cC_0\to W_0$ consists of reducible curves and the
    special fibre $\mathscr{S}_0$ has higher connecting gonality, something which has no effect on the covering gonality,
    but may affect connecting gonality since if a curve in the central fibre decomposes as $C=A+B$, then it could be
    that $(p,q)\in C\times C$ but $p\in A, q\in B$. An example is a family of smooth cubic surfaces (which
    are rationally connected) degenerating to a cone over a plane elliptic curve (which is only rationally chain connected).
    We do not though have an example of a smooth family $\mathscr{S}\to U$ for which this happens. It is in fact not
    even clear whether the above proof or some other argument could give that the central fibre has \textit{chain
    connecting gonality} less than or equal to $d$, meaning two general points in the central fibre are joined by a
    chain of curves, each of which has gonality less than or equal to $d$. Nevertheless from \cite[IV 3.5.3,
    3.11]{kollar} the above all hold for the case $d=1$, namely rational connectedness.
\end{remark}

\section{Preliminaries on the Fano surface of a cubic threefold}\label{fanosurface}

For what follows, the standard references are \cite{cg}, \cite{ak}, \cite{murre}, \cite{tjurin}, \cite{bsd}. For $X
\subset \PP^4$ a smooth projective cubic threefold we know that the Fano surface $S=F(X)$ is a general type surface of
degree $45$ in $\PP^9\cong\PP\left(\bigwedge^2\bC^5\right)$. Moreover the Pl\"ucker embedding is a canonical embedding
of $S$, so that the canonical linear system $|K_S|$ is cut out by hyperplanes (see \cite[10.13]{cg}). 

There are two types of lines in $X$ depending on the image of the polar mapping (or equivalently on the decomposition of
their normal bundle). In particular, we say that a line $\ell\subset X$ is of \textit{second type} if there exists a
2-plane in $\PP^4$ which is tangent to $X$ along $\ell$; otherwise $\ell\subset X$ is said to be of \textit{first type}.
We recall further that given a line $\ell\subset X$, the \textit{incidence divisor} $D_\ell\subset S$ is the curve
parametrising lines in $X$ that meet $\ell$. For a fixed general line $\ell$ (in particular one that is not the residual
in the intersection with the tangent 2-plane of a second type line), after blowing up $X$ along $\ell$ we obtain a conic
bundle structure over the $\bP^2$ parametrising planes containing $\ell$. In this $\bP^2$ sits the smooth discriminant
quintic curve of genus six parametrising planes through $\ell$ where the residual conic curve, in the intersection
of the plane with $X$, is the union of two lines. The incidence divisor $D_\ell$ is a natural \'etale two to one cover
(of the quintic), parametrising each of the two residual lines (the construction is generalised to all lines
$\ell\subset X$ in \cite[1.2]{beauville77}, but the quintic is singular if $\ell$ is not general as above). One computes
(see \cite[p9]{bsd}) that for any $\ell$ the divisor $D_\ell$ satisfies $D_\ell^2=5$ and hence is not a multiple of
another divisor in the N\'eron-Severi group. Also one can prove that $\operatorname{H}^0(S, D_\ell)=1$ for all $\ell$ (see
\cite[1.8]{tjurin}) and that $K_S$ is linearly equivalent to $D_{\ell_1} + D_{\ell_2} + D_{\ell_3}$ for
$\ell_1,\ell_2,\ell_3$ the intersection of $X$ with a 2-plane, see \cite[10.9]{cg}. On the other hand, if $X$ is a
general cubic threefold, then the Picard number of $S$ is one and its N\'eron-Severi group is generated by $D_\ell$ for any
line $\ell$ (see e.g.\ \cite[11]{roulleau1230}). Note however that for $\ell,\ell'$ two distinct lines, the incidence
divisors $D_\ell, D_{\ell'}$ although algebraically equivalent (since they lie in a family parametrised by $S$), differ
by an element of $\Pic^0(S)$ which is five dimensional and isomorphic to the intermediate Jacobian of $X$. 

The locus $D_0\subset S$ parametrising lines of second type is linearly equivalent to the very ample $2K_S$
\cite[10.21]{cg}, and if $X$ is general it is also smooth \cite[1.9]{murre}, which implies that at least in the case
where $X$ is general (namely the case where the Picard number $\rho=1$), $D_0$ is a smooth irreducible curve.
Interesting further analysis of the types of singularities of the discriminant curve and its corresponding quintic, in
the case where $\ell$ is of second type, appears in \cite[p6]{bsd}. Additionally, these and many further results are
generalised to arbitrary characteristic in \cite{murre}.

The first part of the following lemma is standard whereas we could not find a proof of the second in the literature.
Both will be key in later sections of this paper so we include proofs. In what follows, for a hypersurface
$X\subset\PP^N$, we define the \textit{Hessian} $H\subset\PP^N$ as the vanishing locus of the determinant of the matrix
of second derivatives of the equation defining $X$. See \cite{dolgachev} for basic properties of Hessians of
hypersurfaces. We also recall that an \textit{Eckardt point} is a $x\in X$ such that there are infinitely many lines
through $x$ contained in $X$. From \cite[8]{cg}, \cite{roulleauelliptic} there are only finitely many Eckardt points (up
to 30), whereas the general cubic threefold has none.

\begin{lemma}\label{triplepointisspecial}
    Let $X$ be a smooth cubic 3-fold. Then
    \begin{enumerate}
        \item If $x\in X$ is not an Eckardt point, then through $x$ there pass six lines contained in $X$, counted with
        appropriate multiplicities. Moreover if $x\in X$ is general (cf.\ \cite[1.18, 1.19]{murre}), the six lines are
        distinct.
        \item Let $x\in X$ not an Eckardt point. Then $x$ is in the Hessian hypersurface of $X$ if and only if three of
        the above six lines are coplanar (i.e.\ contained in a 2-plane). In this case, the remaining three lines are
        contained in a second 2-plane.
    \end{enumerate}
\end{lemma}
\begin{proof}
    We may assume that $x=[1,0,0,0,0]$ and the tangent hyperplane at $x$ is given by $x_4=0$. The equation of $X$ can be
    then written in the form $F=x_0^2x_4+x_0 Q(x_1,x_2,x_3,x_4)+C(x_1,x_2,x_3,x_4)$, with $Q$ homogeneous quadratic and
    $C$ homogeneous cubic (cf.\ \cite[p307-308]{cg}). 

    For the first claim, the lines through $x$ are parametrized by their intersection point $[0,a_1,a_2,a_3,a_4]$ with
    the hyperplane $x_0=0$. Such a line is given by parametric equations: $x_0=s$, $x_i=ta_i$, $i=1,\ldots ,4$. The line
    is contained in $X$ if the polynomial $f (s,t)=s^2ta_4+st^2Q (a_1,a_2,a_3,a_4)+t^3C (a_1,a_2,a_3,a_4)$ vanishes
    identically in $s$ and $t$. This is equivalent to 
    \[ a_4=0,\;\; Q (a_1,a_2,a_3,a_4)=0,\; \; C (a_1,a_2,a_3,a_4)=0.
    \]
    Therefore the lines through $x$ which are in $X$ correspond to the points $[a_1,a_2,a_3]\in \bP^2$ which satisfy $ Q
    (a_1,a_2,a_3,0)=0$ and $C(a_1,a_2,a_3,0)=0$. These are the intersection points of a conic and cubic in the plane. If these
    have a common component, we have infinitely many lines through $x$ and thus $x$ is an Eckardt point. Otherwise
    there are six intersection points (counted with multiplicities) and we therefore have six lines through $x$ counted
    appropriately.

    To prove the second claim, by \cite[5.9]{cg} the point $x$ is not in the Hessian hypersurface if and only if
    the tangent hyperplane section $V_x:= T_xX \cap X$ has an ordinary double point at $x$. Continuing with the
    assumptions of the first part, $V_x$ is a surface in $\bP^3 \cong {\mathbb V}(x_4)$ given by the equation
    $G(x_0,x_1,x_2,x_3)=x_0Q(x_1,x_2,x_3,0)+C(x_1,x_2,x_3,0)$. That $x$ is an ordinary double point is equivalent to
    (see \cite[5.6]{cg})
    \[ {\rm det} \left( {\partial^2 G \over \partial x_i \partial x_j }|_{ (1,0,0,0,0)},\; {1\leq i,j\leq 3}\right) \neq 0 .
    \]
    Observe now that for $1\leq i,j\leq 3$ we have ${\partial^2 G \over \partial x_i \partial x_j }|_{ (1,0,0,0,0)}=
    {\partial^2 Q \over \partial x_i \partial x_j }$ (note that since $Q$ is homogeneous of degree two this is a
    constant). Therefore the above condition is equivalent to 
    \[ {\rm det} \left( {\partial^2 Q \over \partial x_i \partial x_j },\; {1\leq i,j\leq 3}\right) \neq 0 .
    \]
    which is equivalent to ${\mathbb V}(Q(x_1,x_2,x_3,0))$ being a smooth conic in $\bP^2$. As we saw in the first
    part, the six lines through $x$ correspond to the six intersection points, counted with multiplicities, of the plane
    conic ${\mathbb V}(Q(x_1,x_2,x_3,0))$ with the plane cubic ${\mathbb V}(C(x_1,x_2,x_3,0))$ (we know that they do not
    have a common component because $x$ is not an Eckardt point). The condition that ${\mathbb V}(Q(x_1,x_2,x_3,0))$ is
    smooth is then equivalent to saying that no three of these intersection points lie on a line, otherwise the conic
    contains the line and hence is singular. Conversely if the conic is singular then it is the union of two lines
    and therefore the six intersection points form two triples of coplanar points. 
\end{proof}

One can even say in which cases the above six lines through a point are distinct.

\begin{fact}(\cite[10.18]{cg}, \cite[1.18, 1.19]{murre}) \label{factlines}
    \begin{enumerate}
        \item If $\ell$ is a line of first type and $x\in\ell$, then $x$ is not an Eckardt point and the line $\ell$
        counts with multiplicity one as one of the six lines through $x$. If $\ell$ is a general line of first type,
        through the general point of $\ell$ there pass six distinct lines.
        \item If $\ell$ is a general line of second type, through the general point of $\ell$ there pass four other
        distinct lines and the line $\ell$ is counted with multiplicity two.
    \end{enumerate}
\end{fact}

\begin{lemma}\label{gonalitydell} 
    Let $X$ a smooth cubic threefold.
    \begin{enumerate}
        \item If $\ell$ is a general line of second type, the incidence divisor $D_\ell\subset S$ is a smooth curve
        of genus $11$ that admits a $g^1_4$.
        \item If $\ell$ is a general line of first type, the incidence divisor $D_\ell\subset S$ is a smooth 
        curve of genus $11$ that admits a $g^1_5$.
        \item If $\ell$ general in $S$ then $D_\ell$ is moreover irreducible, whereas if $X$ is a general cubic
        threefold and $\ell$ any line in $X$ then $D_\ell$ irreducible.
    \end{enumerate}
\end{lemma}
\begin{proof}
    For the second type lines, from the second part of Fact \ref{factlines} we know that from the general point of
    $\ell$ there pass four other lines other than $\ell$. We define a map $D_\ell\dashrightarrow\ell$ taking the
    generic point $\ell_s\in D_\ell$ to the unique point of intersection with $\ell$. This extends to a morphism, giving
    a $g^1_4$. 

    For the first type lines, again the map $D_\ell\dashrightarrow \ell$ sending the generic point $\ell_s\in D_\ell$
    to the point of intersection of $\ell_s$ with $\ell$ extends to a morphism, and from Fact \ref{factlines}
    gives a $g^1_5$ on $D_\ell$.

    In both the above cases smoothness (but possible disconnectedness) follows from the fact that the corresponding
    plane quintic curve is smooth (see \cite[1.2]{beauville77}). For irreducibility, from \cite[1.25(iv)]{murre} it
    follows that if $\ell$ is in the open subset of lines of first type which are not residual to a line of second type
    and contained in a smooth hyperplane section of $X$, then $D_\ell$ is irreducible. On the other hand, as mentioned
    above if $X$ is general the Picard number is one and the N\'eron-Severi group is generated by $D_\ell$ for any $\ell$,
    so in particular $D_\ell$ is irreducible.
\end{proof}

\begin{remark}\label{remarkgonalityincidencedivisor}
    We will show in Corollary \ref{corincidencediv} that for the general $X$, the curves $D_\ell$ (for $\ell$ of first
    type as above) are in fact of gonality five. Recalling that $D_\ell$ is an \'etale two to one cover of a plane
    quintic, we point out that a general such cover has instead gonality six \cite[p48]{clb}.  Similarly the gonality of
    the incidence divisor of a general second type line will be four.
\end{remark}

We now note that some sections of $K_S$ come from $2$-planes in $\bP^4$. The construction is as in the following.

\begin{lemma} \label{planesgivingsection} (Fano, \cite[10.3]{cg})
    For $K=\bP^2\subset\bP^4$ a plane, the set $D_K := \{s\in S : \ell_s\cap K\neq \emptyset\}$ is a section of the
    canonical divisor $K_S$. 
\end{lemma}

Note that $\operatorname{H}^0(S, K_S)$ is ten dimensional, whereas there is only a six dimensional family of 2-planes in
$\bP^4$, so the correspondence is not bijective. Nevertheless, in later sections we will use such sections to separate
points on $S$ using hyperplanes in the Pl\"ucker embedding.

We now list some further facts which we will use in later sections. 

\begin{fact}(\cite[p337]{cg}, \cite[8]{roulleauelliptic})\label{factnortlcurves}
If $X$ smooth the Albanese morphism $S\to\operatorname{Alb}(S)$ is an embedding and hence we obtain that every morphism
$\PP^1\to S$ is constant.
\end{fact}

\begin{fact}(\cite[4, 11]{roulleauelliptic})\label{roulleauellcurves}
The cone of lines through an Eckardt point is parame\-trised by an elliptic curve $E\subset S$, and conversely every $E$
gives rise to an Eckardt point. Hence there are at most finitely many elliptic curves in $S$, and for general $X$ there
are none.
\end{fact}

\begin{lemma}\label{ellcurveembedded}
    Let $X$ be a smooth cubic threefold and $S$ its Fano surface. Any non-constant morphism $E\to S$ from a smooth
    elliptic curve has image a smooth elliptic curve. 
\end{lemma}
\begin{proof}
    From Fact \ref{factnortlcurves} we have that $S$ is embedded in its Albanese $\operatorname{Alb}(S)$, so in
    particular we get an induced morphism $E\to \operatorname{Alb}(S)$. This is necessarily a homomorphism of abelian
    varieties up to translation and so its image is a one dimensional abelian subvariety of $\operatorname{Alb}(S)$,
    hence smooth.
\end{proof}

\section{Covering families and birational very ampleness}

In \cite{bdelu} the following condition $\BVA_p$ is introduced and its connection to covering gonality is studied.

\begin{definition}
    Let $S$ be an irreducible projective variety. A line bundle $L$ on $S$ \textit{satisfies condition $\BVA_p$} for
    $p\geq0$ if there exists a closed subset $Z\subset S$ depending on $L$ such that for every $0$-dimensional subscheme
    $\xi$ of length $p+1$ with support disjoint from $Z$, the following restriction map is surjective
    $$\operatorname{H}^0(S, L) \to \operatorname{H}^0(S, L\otimes \cO_{\xi}). $$
\end{definition}

This definition says that outside a closed subvariety $Z$, sections of $L$ can ``separate any $p+1$ points'' meaning
that for any subset of $p$ points in a set of $p+1$ points on $S$ not contained in $Z$, there is a section of $L$
vanishing at all points of the subset, but not at the extra point. The relation to covering gonality is given by the
following.

\begin{proposition}(\cite[1.10]{bdelu})\label{bvapimpliescovgon}
    Let $S$ be a smooth projective variety. If $K_S$ satisfies $\BVA_p$ then $\cg(S)\geq p+2$.
\end{proposition}

We will now study the condition $\BVA_p$ for the Fano scheme of lines on a smooth cubic threefold, showing that this
condition is much too strong; cf.\ Proposition \ref{bva1notbva2}.

\begin{lemma}\label{bva1}
    Let $X$ be a smooth cubic threefold and $S:=F(X)$ its Fano surface of lines. Then $K_S$ satisfies $\BVA_1$ and so
    $\cg(S)\geq3$. 
\end{lemma}
\begin{proof}
    Since $K_S$ is very ample it is certainly $\BVA_1$. The claim about the covering gonality now follows from
    Proposition \ref{bvapimpliescovgon}. 
\end{proof}

\begin{remark}
    To obtain the above claim about the covering gonality in a different way but only for the general cubic threefold,
    one can degenerate to a $1$-nodal cubic $X_0$ whose $F(X_0)$ is a non-normal surface with desingularisation $\Sym^2 C$ for
    $C$ a smooth trigonal genus four curve, hence by \cite[1.6]{bastianelli} and Proposition \ref{propgonalitydrops} we get a
    lower bound of three again.
\end{remark}

\begin{proposition}\label{bva1notbva2}
    Let $X$ be a smooth cubic threefold and $S:=F(X)$ its Fano surface of lines. Then through a general point in $S$ (in
    fact one corresponding to a line not contained in the Hessian hypersurface) there pass five lines contained in
    $\rm{G}(2,5)\subset\PP^9$ which are trisecant to $S$. In particular $K_S$ does not satisfy $\BVA_2$. 
\end{proposition}
\begin{proof}
    From Lemma \ref{triplepointisspecial} we know that at the points of $X$ lying on the intersection with the Hessian
    hypersurface, there pass two tuples of three coplanar lines. Since the Hessian has degree five in $\PP^4$, we see
    that a general line $\ell$ in $X$ meets the Hessian in five points and through each of these points there will be
    two residual lines coplanar to $\ell$. This means that the induced line in the Grassmannian (corresponding to lines
    in this plane through the point) will be trisecant to the Fano surface $S$. To summarise, the general point in $S$
    admits five lines through it which are trisecant to $S$ in $\PP^9$, so that no canonical divisor on $S$ passes
    through exactly two of the intersection points between $S$ and one of those lines. 
\end{proof}

Even though $\BVA_2$ does not hold, there are more refined properties of sections of the canonical divisor which will
however apply to our situation. The following is a direct consequence of \cite[4.2, 4.7]{bastianelli}, the ideas going
back further e.g.\ to \cite{lopezpirola}.

\begin{proposition}\label{cb}
    Let $S$ be a smooth projective variety and $\cC/T\to S$ a covering family of $d$-gonal curves. Let $t\in T$ a
    general point and $\xi=p_1+\ldots+p_d$ a $0$-cycle on $S$ obtained as the image of a general effective divisor in the
    $g^1_d$ on $\cC_t$. Then $\xi$ satisfies the Cayley-Bacharach property with respect to sections of $K_S$, meaning
    that for any $1\leq j\leq d$ and any effective canonical divisor $D\in|K_S|$ passing through
    $p_1,\ldots,\hat{p_j},\ldots,p_d$ we have $p_j\in D$.
\end{proposition}

\begin{remark}\label{remsepfibreg1d}
    The above proposition implies that if $S$ is a smooth projective canonically polarised variety and $f:\cC/T\to S$ a
    covering family of $d$-gonal curves, then the image of the general fibre of the $g^1_d$ of the general member
    $\cC_t$ of this family consists of $d$ distinct points on $S$ lying on a $(d-2)$-plane under the canonical
    embedding. This can be proved more directly using geometric Riemann-Roch like in the proof of \cite[1.10]{bdelu}.
    Note however that the property that a number of points span a plane of smaller dimension than expected is strictly
    weaker than the above Cayley-Bacharach property for those points: consider for example four points in $\PP^3$, three
    of which are collinear - they span a 2-plane but do not satisfy the Cayley-Bacharach property with respect to
    hyperplanes.
\end{remark}

In the particular cases we will be interested in, the geometry of the varieties under consideration imposes conditions
on the above configurations of $d$ points.

\section{Lines in special position in $\bP^4$}\label{sectionseppoints}

The following is an auxiliary section containing a technical result on configurations of lines in $\PP^4$. Let $\rm{G}(2,5)$
denote the Grassmannian of lines in $\PP^4$. For a fixed $\Lambda\cong\PP^2\subset\PP^4$ we have an induced divisor
$\sigma_1(\Lambda)\subset \rm{G}(2,5)$ called the first Schubert cycle, whose points correspond to lines meeting $\Lambda$
and whose linear equivalence class is that of the very ample divisor inducing the Pl\"ucker embedding. Note that as
mentioned after Lemma \ref{planesgivingsection}, the Pl\"ucker ample has ten sections, whereas there is only a six
dimensional family of sections of type $\sigma_1$.

\begin{definition}(\cite[3.2]{bastianelli})
    Let $\ell_1,\ldots,\ell_d\subset \PP^4$ be $d$ distinct lines. We say that the $\ell_i$ are in \textit{special
    position} (with respect to 2-planes) if for every $1\leq j\leq d$ and a 2-plane $\Lambda\subset\PP^4$ intersecting
    $\ell_1,\ldots,\hat{\ell_i},\ldots,\ell_d$ we have that $\ell_i$ also meets $\Lambda$. Analogously, we define $d$
    distinct lines in $\bP^3$ to be in special position (with respect to lines).
\end{definition}

Recall that the Pl\"ucker embedding on $\rm{G}(2,5)$ induces the canonical embedding on the Fano scheme $S$ of a cubic
threefold. The significance of the above definition in our context comes from Proposition \ref{cb}, which implies that
if $\cC/T\to S$ a $d$-gonal covering family of curves of the Fano scheme of lines, then the image of the general fibre
of the $g^1_d$ consists of $d$ points corresponding to lines in $\PP^4$ which are in special position.

We note that in order to check whether a set of distinct points $\{[\ell_1],\ldots,[\ell_d]\}\subset \rm{G}(2,5)$ corresponds
to lines which are \textit{not} in special position in $\PP^4$, it suffices to find a hyperplane $\sigma_1$ which
contains $d-1$ of them but not the last one.

\begin{proposition} \label{propseppoints}
    \begin{enumerate}
        \item Three distinct points $[\ell_1], [\ell_2], [\ell_3]\in \rm{G}(2,5)$ correspond to lines in special position if
        and only if the $\ell_i$ lie in a 2-plane in $\PP^4$ and all meet at a point.
        \item Four distinct points $[\ell_1], \ldots, [\ell_4]\in \rm{G}(2,5)$ correspond to lines in special position if and
        only if one of the following configurations occurs:
            \begin{enumerate}
                \item The lines $\ell_1,\ldots,\ell_4$ lie in a 2-plane and meet at a point.
                \item The lines $\ell_1,\ldots,\ell_4$ lie in a 2-plane and no three of them meet at a point.
                \item The lines $\ell_1,\ldots,\ell_4$ span a 3-plane $H$ and lie on the same ruling of a smooth quadric
                    $Q\subset H$.
                \item The lines $\ell_1,\ldots,\ell_4$ span a 3-plane, $\ell_1,\ell_2$ meet at a point $P$,
                    $\ell_3,\ell_4$ meet at a point $P'\neq P$ and the induced 2-planes $\langle\ell_1,
                    \ell_2\rangle$, $\langle \ell_3,\ell_4\rangle$ meet along the line $PP'$.
                \item The lines $\ell_1,\ell_2,\ell_3,\ell_4$ span a 3-plane, all meet at a point and no three lie in a
                    2-plane.
            \end{enumerate}
    \end{enumerate}
\end{proposition}

\begin{figure}
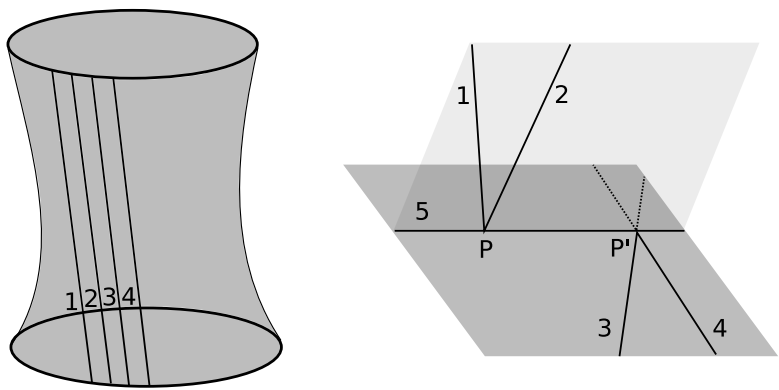

\psset{xunit=.4pt,yunit=.4pt,runit=.4pt}


\caption{Configurations (c), (d) of Proposition \ref{propseppoints} respectively.}
\label{sketch}
\end{figure}

\begin{proof}
For the first part of three distinct lines, the condition that they are in special position in $\PP^4$ is equivalent to
the points $[\ell_i]\in \rm{G}(2,5)$ all lying on a line $L\subset \PP^9$. Since $\rm{G}(2,5)$ is cut out by quadrics,
$L$ has to be contained in the Grassmannian by B\'ezout, which is equivalent to the condition of the proposition.

For the second part, note first that it is easy to check that configurations (a)-(e) all correspond to lines which are
in special position. Conversely, assume that $\ell_1,\ldots,\ell_4$ are in special position. By Theorem
\cite[3.3]{bastianelli} four lines in $\PP^4$ that are in special position must necessarily be contained in a 3-plane
$H$ and moreover they must be in special position there with respect to lines, since then any line $\ell$ meeting three
of the four lines has to meet all of them, otherwise the span of $\ell$ and a point outside $H$ would be a 2-plane
contradicting the special position of the four lines.  

Assume first that they span a 2-plane. If no three of them are concurrent then we are in configuration (b). If instead
three of the lines, say $\ell_1,\ell_2,\ell_3$ meet at a point $P$, then a general line in $H$ containing $P$ must intersect
$\ell_4$ necessarily at $P$, so we are in configuration (a).

Assume next that the lines span the 3-plane $H$. If they are skew then \cite[3.6]{bastianelli} gives that we are in
configuration (c). Assume henceforth that two of the lines, say $\ell_1,\ell_2$, meet at a point $P$. Then any line
meeting $\ell_3$ and passing through $P$ must intersect $\ell_4$. If $\ell_3$ passes through $P$ then $\ell_4$ must also
pass through $P$. Moreover no three of them can lie on a 2-plane $\Lambda$ because a line on $\Lambda$ not passing
through $P$ intersects the three lines on the plane but not the fourth. This implies we are in configuration (e). If on
the other hand $\ell_3$ does not pass through $P$ then $\ell_4$ has to be contained in the plane spanned by $P$ and
$\ell_3$. Moreover the point of intersection of $\ell_3$ with $\ell_4$ must lie on the intersection line $\ell_5$ of the
planes spanned by $\ell_1,\ell_2$ and $\ell_3,\ell_4$ respectively; otherwise, if we took a line not equal to $\ell_5$
in the first of those planes passing through the intersection of $\ell_3$ with $\ell_5$, then this would intersect
$\ell_1,\ell_2,\ell_3$ but not $\ell_4$. This implies we are in configuration (d).
\end{proof}

Making a complete list of all possible configurations of five lines in special position in $\bP^4$ is cumbersome, so we
will soon restrict ourselves to a more particular situation, which will be precisely what is required in the case
of the Fano scheme of lines of a cubic threefold. First, the following gives a structure for the most general
configuration of five lines in special position.

\begin{proposition}\label{propfivelinesuniquesecant}
    Suppose $\ell_1,\ldots,\ell_5$ are five lines in ${\mathbb P}^4$ in special position with respect to 2-planes.
    Assume that any pair of lines spans a 3-plane and any triple of lines spans the whole space. Then there is a
    unique common secant line to the above five lines. 
\end{proposition}
\begin{proof}
    By the lemma which follows, the lines $\ell_1,\ldots, \ell_4$ lie on the ruling of a smooth rational normal cubic
    scroll $S_{1,2} \subset\bP^4$. Hence any line meeting $\ell_1,\ldots, \ell_4$ is contained in $S_{1,2}$ and does not
    belong to the ruling. By \cite[8.20]{harris} such a line is unique and coincides with the directrix of the scroll
    and it must intersect also $\ell_5$ since the lines are in special position.
\end{proof}

\begin{lemma}
    Under the assumptions of the proposition, any four of the above lines are lines of the ruling of a smooth
    rational normal cubic scroll in ${\mathbb P}^4$ and the fifth line intersects the directrix of the scroll.
\end{lemma}
\begin{proof}
    Recall from \cite[p92-93]{harris} that in ${\mathbb P}^4$ we have, up to projective equivalence, two types of
    rational normal scrolls: the smooth cubic scroll $S_{1,2}$ and the singular cone $S_{0,3}$, which by
    \cite[p525]{griffithsharris} are the only irreducible non-degenerate surfaces of degree three in ${\mathbb P}^4$.
    Moreover, four lines in the ruling of $S_{1,2}$ satisfy the above genericity conditions and the projection of
    $S_{1,2}$ from a point outside the directrix is an $S_{1,1}$, i.e.\ a smooth quadric in ${\mathbb P}^3$, see
    \cite[8.20]{harris}. Note also, that in case $\ell_5$ is a line in the ruling of the scroll, the converse is also
    true. In fact, lines in the ruling of $S_{1,2}$ are in special position with respect to 2-planes, by
    \cite[3.7]{bastianelli}.
        
    Take the projection from a general point $P$ of $\ell_5$ to a general hyperplane $H$ not containing $P$ or any of
    the five lines. By the assumptions we claim now that the images $\ell'_i$ of the $\ell_i$'s, for $i=1,2,3,4$, are
    skew lines in ${\mathbb P}^3$ which are in special position with respect to lines. For the former claim, if two of
    the lines, say $\ell'_1, \ell'_2$, intersected then $P$ would belong in the secant variety (the span) of $\ell_1,
    \ell_2$. Since $P$ is generic this would imply that $\ell_5$ is in the span of $\ell_1, \ell_2$ which contradicts
    the assumptions. For the latter claim, see the proof of theorem \cite[3.3]{bastianelli}. In Proposition
    \ref{propseppoints} (2) we classified such configurations of lines, since as seen in the statement, four lines in
    ${\mathbb P}^4$ in special position with respect to 2-planes are lines in a hyperplane in special position with
    respect to lines. Since the $\ell'_i$'s do not intersect, we are in case (c) i.e.\ they are lines in a ruling
    ${\mathcal R}_1$ of a smooth quadric surface $Q$. We denote by $m_t$ the lines in the other ruling ${\mathcal R}_2$,
    $t\in {\mathbb P}^1$. For each $t$, we take the 2-plane $\Pi_t$ spanned by $m_t$ and the point $P$. This intersects
    the lines $\ell_1,\ldots, \ell_4$ in four points and we take a conic curve $C_t$ passing through those four
    points and the point $P$. Note that by the assumptions, any triple of lines $\ell_1,\ldots,\ell_5$ has at most one
    secant line, and therefore for the general $t$, no three of those five points are on a line, which implies that
    $C_t$ does not contain a line. Hence $C_t$ is smooth and is the unique conic passing through $P$ and the four points
    of intersection of $\ell_1,\ldots,\ell_4$ and $\Pi_t$. Define $S$ to be the closure of the total space of the family
    of the smooth $C_t$'s. This is an irreducible non degenerate surface that contains the lines $\ell_1,\ldots,\ell_4$
    and the point $P$. We claim that it is of degree three, which will imply that it is the scroll $S_{1,2}$ (note that
    since it contains four skew lines, it cannot be the scroll $S_{0,3}$). 

    Indeed, taking a general $s\in \mathcal{R}_2$ and noting that $\ell_1'$ and $\Pi_s$ meet, let $H_s$ be the 3-plane
    they span. For $t\neq s$ another point in $\mathcal{R}_2$, the plane $\Pi_t$ intersects $H_s$ exactly at the line
    $\ell_{PQ_t}$, with $Q_t$ the point of intersection of $\ell'_1$ with $m_t$: indeed this line is certainly contained
    in the intersection and if it were not the intersection, then $\Pi_t$ would be contained in $H_s$. As the lines
    $m_t, m_s$ are skew, they span the hyperplane $H$ which implies that $H_s$ would contain the span of $H$ and $P$
    which is the whole space, hence a contradiction, proving $\Pi_t\cap H_s = \ell_{PQ_t}$. Next, by moving $t$ we see
    that the intersection of the family of $\Pi_t$'s with $H_s$ is the union of the 2-plane $\Pi_s$ and the 2-plane
    spanned by $\ell'_1$ and the point $P$. Note that the total space of the family of $\Pi_t$'s is a rank four quadric
    (the quadric cone over $Q$ with vertex $P$) and therefore the intersection with a hyperplane is a surface of degree
    two. So the intersection between the quadric cone and $H_s$ coincides with $\Pi_s\cup\langle \ell_1',P\rangle$. But
    then the intersection of the two planes $\Pi_s$ and $\langle \ell'_1, P\rangle$ with $S$ is the union of the line
    $\ell_1$ and the conic $C_s$, which is overall a curve of degree three. Therefore $S$ has degree three, is
    irreducible and so as mentioned above is the rational normal scroll $S_{1,2}$ which projects from $P$ to the quadric
    surface $Q$.  Moreover $P$ is not contained in the directrix of $S$ because the projection of $S_{1,2}$ from a point
    of the directrix is the singular quadric $S_{0,3}$, see \cite[8.20]{harris}, whereas we already established that $Q$
    is smooth. Moreover the line $\ell_5$ has to intersect $S$ also at a point of the directrix $L$ of the scroll (e.g.\
    $\ell_5$ might be a line in the ruling of $S$). Indeed, otherwise consider the span of the two skew lines $L$ and
    $\ell_5$. This is a hyperplane and taking a point $P'$ outside this, the 2-plane spanned by $P'$ and $L$ intersects
    $\ell_1,\ldots,\ell_4$ but not $\ell_5$ which contradicts their special position.
\end{proof}

From now on, in order to simplify our analysis a bit, we will assume in addition that in the set of five lines,
no line or group of lines can be distinguished from the rest. More formally, for a set of lines
$\mathcal{A} := \{\ell_1,\ldots,\ell_5\}$ define
\begin{eqnarray*}
n(\ell_i)&=&\#\{\ell_j : j\neq i, \ell_i\cap\ell_j\neq\emptyset\}\\
m(\ell_i)&=&\max\{k : \ell_i\text{ and }k\text{ other lines in }\mathcal{A}\text{ span a 2-plane}\}.
\end{eqnarray*}

We will assume $\mathcal{A}$ satisfies the following property

\begin{longtable}{l p{13cm}}
(R) & We have $n(\ell_i)=n(\ell_j), m(\ell_i)=m(\ell_j)$ for any $i,j\in\{1,\ldots,5\}$.
\end{longtable}

The significance of the above in our situation is that if the lines correspond to the points of a general fibre of a
five-to-one {\em irreducible} cover of curves (in some parameter space like the Grassmannian), then they satisfy
property (R), e.g.\ by a monodromy argument as outlined in \cite[2.4]{bastianelli}.
 
\begin{proposition}\label{propfivelinestwospanaplane}
    Suppose we have five lines in ${\mathbb P}^4$ in special position with respect to 2-planes which also satisfy
    property (R). Assume that a pair of them spans a 2-plane. Then all of them are coplanar or all of them pass
    through the same point. 
\end{proposition}

\begin{proof}
    In view of assumption (R), let $n=n(l_i)$ and $m=m(l_i)$. Then $1\leq m\leq n\leq 4$ and in addition, $n\geq 2$ as 
    the number of lines is odd. Notice that if four lines are concurrent then also the fifth must pass through the same
    point since they are in special position; otherwise considering the 2-plane spanned by the fifth line and the point of
    intersection of the first four, we see then that any 2-plane in $\mathbb{P}^4$ which intersects this plane in the
    aforementioned point of intersection contradicts special position. Moreover if $m\geq 3$ then $m=4$ by
    \cite[3.4]{bastianelli} (which in fact holds for any number of lines).

    We claim that if $n=4$, i.e.\ if any line meets all the others, then the assertion holds. Suppose indeed that the
    lines are not concurrent, then it is elementary to see that there exists a point through which exactly two lines
    pass, say $\ell_1$ and $\ell_2$. Since $n=4$, $\ell_3$ intersects $\ell_1$ and $\ell_2$ in distinct points, so
    $\ell_3$ lies in the plane $\langle\ell_1,\ell_2\rangle$ and the same happens for $\ell_4$ and $\ell_5$. Thus the
    five lines are coplanar. 

    The case $n=3$ cannot occur, as if any line meets exactly three of the others, then taking the hyperplane spanned by
    two non intersecting lines and a point $P$ outside this, we see that the 2-plane spanned by the first line and $P$ will
    not intersect the second line but it will intersect all other lines, contradicting special position.

    Thus we may assume $n=2$ and we need to only consider the cases $m=1$ and $m=2$. The case $m=2$ does not occur. Indeed,
    if $\ell_1, \ell_2, \ell_3$ span a plane, then so also do $\ell_4, \ell_5$, and the plane $\langle\ell_4, \ell_5\rangle$
    must contain another line, say $\ell_1$. Thus $\ell_1$ intersects all the other lines violating the condition that $n=2$.
    Finally we aim to rule out the case $n=2$ and $m=1$, i.e.\ when any line meets two of the others and no three lines
    are coplanar. Notice that if three of the lines were concurrent then they would not intersect the remaining lines,
    which would violate the condition $n=2$. Without loss of generality suppose thus that $\ell_1\cap \ell_5\neq
    \emptyset$ and $\ell_i\cap \ell_{i+1}\neq \emptyset$ for any $1\leq i \leq 4$, with $P_{ij}:=\ell_i \cap \ell_j$
    with $i<j$, whereas all other intersections are empty (the above configuration resembles the boundary lines of a
    pentagon in $\mathbb{P}^4$). Notice that the 2-plane $\langle \ell_3, \ell_4 \rangle$ intersects $\ell_2,\ldots,
    \ell_5$, so from special position there must exist a point $Q_1 \in \ell_1 \cap \langle \ell_3, \ell_4 \rangle $.
    Moreover $Q_1$ must lie on the line $\langle P_{23},P_{45} \rangle$ by the special position assumption. Analogously
    there exists a point $Q_4 \in \ell_4 \cap \langle \ell_1, \ell_2 \rangle$ lying on the line $\langle P_{23},P_{15}
    \rangle$. Notice that the points $Q_1, P_{23}, Q_4$ are not collinear and they all lie on both planes $\langle
    \ell_1, \ell_2 \rangle$ and $\langle \ell_3, \ell_4 \rangle$. Therefore these two planes coincide which contradicts
    the assumption $m=1$.
\end{proof}

\begin{proposition}\label{propfivelinesthreespanahyperplane}
    Suppose we have five lines in special position in ${\mathbb P}^4$ which also satisfy property (R).
    Suppose that no two of them intersect but some triple of them, say $\ell_1, \ell_2, \ell_3$, spans a hyperplane
    $H$. Then all five lines are contained in $H$. Let $Q\cong {\mathbb P}^1 \times {\mathbb P}^1$ be the secant surface
    of $\ell_1, \ell_2, \ell_3$ in $H$ and let ${\mathcal R}_1$ be the ruling containing $\ell_1, \ell_2 ,\ell_3$. We
    then have one of the following cases:
    \begin{enumerate}
        \item The five lines are all lines in the ruling ${\mathcal R}_1$. 
        \item The lines $\ell_4, \ell_5$ are not contained in $Q$ and in this case the five lines possess two or
        one common secant lines.
    \end{enumerate}
\end{proposition}

\begin{proof}
    If neither of $\ell_4, \ell_5$ is contained in $H$, then from special position both have to intersect $H$ at the
    same point which contradicts the assumptions. Therefore one of them is contained in $H$ and hence the other too by
    \cite[3.4]{bastianelli}. As a consequence we have five skew lines in $H$ which therefore have to be in special
    position with respect to lines in $H$. Note that the only lines that intersect $\ell_1, \ell_2, \ell_3$ are lines in
    the second ruling ${\mathcal R}_2$ of the quadric $Q$ (see Proposition \ref{propseppoints}). We see now that case
    (1) above can indeed occur. Suppose now that one of $\ell_4, \ell_5$, say $\ell_4$, is not a line in the ruling
    ${\mathcal R}_1$. If $\ell_4$ is contained in $Q$ then it is a line of the second ruling ${\mathcal R}_2$ which
    means it intersects $\ell_1$, contradicting the assumptions. Analogously $\ell_5$ is also not contained in $Q$. It
    follows that $\ell_4$ intersects the degree two surface $Q$, which is the secant variety of $\ell_1,\ell_2,\ell_3$,
    at one or at two distinct points and accordingly this implies that there is exactly one or two secants to the lines
    $\ell_1,\ldots,\ell_4$ - that is, the lines of the ruling $\mathcal{R}_2$ of the quadric going through the above
    points. From special position these have to intersect $\ell_5$ also, and we are thus in case (2). 
\end{proof}

The above complete the possibilities for five lines in special position in $\bP^4$ which satisfy property (R).

\section{The Klein cubic threefold}

The vanishing locus $X\subset \PP^4$ of the smooth cubic equation

$$ F = x_0^2x_1 + x_1^2x_2 + x_2^2x_3 + x_3^2x_4 + x_4^2x_0 $$
is called the \textit{Klein cubic}. It has especially nice properties as seen for example in \cite{adler},
\cite{roulleauklein}. To name a few, its Fano surface $S$ is of maximal Picard number $\rho=h^{1,1}=25$, and its Hessian
has beautiful geometry related to the modular curve $X_0(11)$.

In \cite[Sections 37-39]{adler} the singular locus of the Hessian hypersurface $H\subset\PP^4$ of the equation $F$ of
$X$ is studied, and an explicit resolution $\hat{H}\subset\PP^4\times\PP^4$ of singularities of $H$ is constructed. We
will first show that $X\cap H$ is also resolved by this birational morphism, and that this implies that $X\cap H$ is not
uniruled. First we describe the explicit resolution mentioned above. The Hessian matrix $M(F)$, of second partial
derivatives of $F$, is given for $\textbf{x}=(x_0,\ldots, x_4)$ by
$$
M_{\textbf{x}}(F) = 
2\begin{pmatrix}
    x_1 & x_0 & 0 & 0 & x_4 \\
    x_0 & x_2 & x_1 & 0 & 0 \\
    0 & x_1 & x_3 & x_2 & 0 \\
    0 & 0 & x_2 & x_4 & x_3 \\
    x_4 & 0 & 0 & x_3 & x_0 
\end{pmatrix}.
$$
As in \cite[Section 37]{adler}, if $\textbf{y}=(y_0,\ldots, y_4)^T$ then the vanishing of the system of
five quadrics

$$ \hat{H}:=\{(\textbf{x},\textbf{y}) : M_{\textbf{x}}(F)\cdot\textbf{y} = 0\} \subset \PP^4\times\PP^4 $$
endowed with the first projection $\pi:\hat{H}\to H$ gives the required resolution of $H$. 

\begin{lemma}
    Let $H$ be the Hessian hypersurface of the Klein cubic threefold $X$, let $B=X\cap H$ their intersection and
    $\pi:\hat{H}\to H$ the resolution described above. Then $\bar{\pi}:\hat{B}:=\pi^{-1}(B)\to B$ is a resolution of
    singularities of $B$ and both are irreducible. 
\end{lemma}
\begin{proof}
    The equations of the preimage of $B$ under the resolution $\pi:\hat{H}\to H$ have Jacobian matrix (up to a constant)
    $$
    \begin{tiny}\begin{pmatrix}
    y_1 & y_0 & 0 & 0 & y_4 & x_1 & x_0 & 0 & 0 & x_4 \\
    y_0 & y_2 & y_1 & 0 & 0 & x_0 & x_2 & x_1 & 0 & 0 \\
    0 & y_1 & y_3 & y_2 & 0 & 0 & x_1 & x_3 & x_2 & 0 \\
    0 & 0 & y_2 & y_4 & y_3 & 0 & 0 & x_2 & x_4 & x_3 \\
    y_4 & 0 & 0 & y_3 & y_0 & x_4 & 0 & 0 & x_3 & x_0 \\
    2x_0x_1+x_4^2 & 2x_1x_2+x_0^2 & 2x_2x_3+x_1^2 & 2x_3x_4+x_2^2 & 2x_4x_0+x_3^2 & 0 & 0 & 0 & 0 & 0
    \end{pmatrix}.\end{tiny}
    $$
    A computation in Macaulay2 below shows that $B$ is irreducible by checking it is irreducible on the affine open
    $x_0\neq0$ and then, since it is a complete intersection and hence equidimensional, that its intersection with the
    hyperplane $x_0=0$ is one dimensional (which is much faster than checking whether the ideal of $B$ is prime). 
    It also computes that the above matrix is of maximal rank at every point of the preimage of $X$ and hence $\hat B$ is
    smooth. Since \cite[Section 38]{adler} proves that $H$ is singular along a smooth curve $C$ of degree twenty, one
    can compute that $B$ is singular precisely at the sixty distinct points of intersection of $X$ with $C$, showing
    that $\bar{\pi}:\hat{B} \to B$ is birational.

\noindent\begin{Verbatim}[fontsize=\small]
 R = QQ[x_0..x_4];
 F = x_0^2*x_1 + x_1^2*x_2 + x_2^2*x_3 + x_3^2*x_4 + x_4^2*x_0;
 partialsF = submatrix(jacobian matrix{{F}}, {0..4},{0});
 H=det submatrix(jacobian ideal partialsF, {0..4},{0..4});
 B=(ideal F)+(ideal H); isPrime substitute(B, x_0=>1)
 dim variety (B+(ideal x_0))
 P=variety ideal singularLocus B;
 print degree P; print HH^0(OO_P);
 R = QQ[x_0..x_4,y_0..y_4];
 F = x_0^2*x_1 + x_1^2*x_2 + x_2^2*x_3 + x_3^2*x_4 + x_4^2*x_0;
 Hhat = ideal {x_1*y_0+x_0*y_1+x_4*y_4, x_0*y_0+x_2*y_1+x_1*y_2, 
  x_1*y_1+x_3*y_2+x_2*y_3, x_2*y_2+x_4*y_3+x_3*y_4, x_4*y_0+x_3*y_3+x_0*y_4};
 Bhat = ideal(F)+Hhat; partials = jacobian Bhat;
 for i from 0 to 4 do (for j from 0 to 4 do (
 print dim (Bhat + minors(6, partials) + ideal (x_i-1, y_j-1));))
    \end{Verbatim} 
\end{proof}

\begin{proposition}\label{kleinnotuniruled}
    If $X$ is the Klein cubic threefold and $H$ its Hessian hypersurface, then $X\cap H$ is not uniruled. 
\end{proposition}
\begin{proof}
    From the previous lemma, since $\hat{B}=\pi^{-1}(B)$ is a section of $\pi^*\cO_{H}(3)$, we compute by adjunction
    that $K_{\hat{B}} = K_{\hat{H}}|_{\hat{B}} + \bar\pi^*\cO_{H}(3)$. Now as pointed out in \cite[p139]{adler}, since
    $\hat{H}$ is the smooth intersection of five quadrics in $\PP^4\times\PP^4$, it is a smooth Calabi-Yau threefold,
    which implies $K_{\hat{H}}$ is trivial. In particular $K_{\hat{B}}$ is a big line bundle as it is the pullback of
    $\cO_H(3)$ under a birational morphism. Hence, since $B$ is birational to a smooth general type surface, it cannot
    be covered by rational curves.
\end{proof}

\begin{corollary}\label{generalhessiannotuniruled}
    Let $X$ be a general cubic threefold and $H$ its Hessian hypersurface in $\PP^4$. Then $X\cap H$ is not uniruled.
\end{corollary}
\begin{proof}
    Since from Proposition \ref{kleinnotuniruled} for the Klein cubic this is not the case and we know from
    \cite[IV.1.8]{kollar} that uniruledness is a closed condition in families, the result follows.
\end{proof}

\begin{remark}
    Note that if $X$ is the Fermat cubic threefold, then the Hessian is a union of five hyperplanes, so in particular
    its intersection with $X$ is the union (along Fermat cubic curves) of five smooth Fermat cubic surfaces so all its
    irreducible components are certainly uniruled.
\end{remark}

\section{Proofs of the theorems}\label{sectionproof}

The strategy of the proofs is the following. From Remark \ref{remsepfibreg1d} we know that if $\cC/T\to S$ is a covering
family of $d$-gonal curves, then the general fibre of the $g^1_d$ of the general member of this family will lie on a
$d-2$ plane under the canonical embedding. From Lemmas \ref{bva1}, \ref{gonalitydell} we know that the covering gonality
is at least three and at most four, in particular we need to only rule out the case of covering gonality three where
$X\subset\PP^4$ is a general cubic threefold. We will moreover carry out analysis in the case of tetragonal and
pentagonal covering families, leading to the uniqueness statements of Proposition \ref{covgonfour}. Since from
Proposition \ref{propseppoints} we have a complete classification of the possible configurations of three or four lines
which are in special position, we need only prove that the cases in which these configurations occur as the generic
fibre of the $g^1_d$ of a general member of $\cC/T$, can not occur for a covering family of the Fano surface of a
general cubic threefold except for the one family listed in Proposition \ref{covgonfour}. Certain cases can be ruled out
automatically, as for example four lines can not be contained in the intersection of a cubic with a 2-plane since a
smooth cubic threefold contains no 2-planes, but for more complicated ones we do this by degenerating to the Klein cubic
threefold.

\subsection{The proof of Theorem \ref{covgongen}}

As indicated in the discussion at the beginning of Section \ref{sectionproof}, to show that the general cubic threefold
has Fano surface with covering gonality at least four we need to exclude configuration (1) of Proposition \ref{propseppoints}. From
Lemma \ref{3.1}, if (1) does occur for the general fibre of the $g^1_3$ of the general member $\cC_t$ of a covering
family $\cC/T\to S$, then the intersection $H\cap X$ has to be uniruled. From Corollary \ref{generalhessiannotuniruled}
we thus obtain that $\cg(S)\geq4$ for the general $X$ and the theorem is proven. 

\begin{lemma}\label{3.1}
    Let $S$ be the Fano surface of a smooth cubic threefold $X$ and $H$ the Hessian of $X$. Assume that $\cC/T\to S$ is
    a covering family of trigonal curves. Then the intersection of $X$ with its Hessian hypersurface is a uniruled
    (possibly singular) surface. 
\end{lemma}
\begin{proof}
    We know from Proposition \ref{propseppoints}, Remark \ref{remsepfibreg1d} and the fact that there can only be
    finitely many Eckardt points (and hence at worst a divisor in $S$ parametrising lines through Eckardt points) that
    the general fibre of the $g^1_3$ of the general member $\cC_t$ will consist of three distinct points
    $[\ell_1],[\ell_2],[\ell_3] \in S$ so that the $\ell_i$ are coplanar and pass through the same point. From the
    second part of Lemma \ref{triplepointisspecial}, this point of intersection is contained in the intersection $X\cap
    H$ with the Hessian hypersurface of $X$. This way we obtain a map $\PP^1\dashrightarrow X\cap H$ which extends to a
    morphism, where a general point $y\in\PP^1$ is sent to the unique intersection point $x\in X$ of the lines
    $\ell_1,\ell_2,\ell_3$ parametrised by the fibre over $y$ of the trigonal map. Varying the point $t\in T$, by
    construction we necessarily obtain a one dimensional family of rational curves in $X\cap H$ since there are only
    finitely many points in $X$ with infinitely many lines through them (Eckardt points). This gives the result. 
\end{proof}

\begin{remark}
    Alternatively, one could conclude from the above argument that the Fano surface of the Klein cubic does not have
    covering gonality three, and then degenerate covering gonality instead of uniruledness from Proposition 
    \ref{propgonalitydrops} to obtain the result for the general $X$.
\end{remark}

\subsection{The proof of Proposition \ref{covgonfour}}

To exclude covering gonality four for families other than the second type incidence divisors, we work through the
various configurations of Proposition \ref{propseppoints}. Configurations (a), (b) of the second part of Proposition
\ref{propseppoints} can be excluded from the fact that the lines are contained on a smooth cubic threefold. On the other
hand we will prove that the only covering family for which configuration (e) occurs is the family of incidence divisors
of second type lines of Lemma \ref{gonalitydell}. Excluding (c), (d) relies on geometric properties of the general $X$
and the fact that such configurations of points would have to appear as a general fibres of the $g^1_4$ of the general
member of a tetragonal covering family.

Indeed if one of configurations (a), (b) holds for four points in $S$, then since every 2-plane in $\PP^4$
meeting $X$ in four lines has to be contained in $X$, we obtain a contradiction to the Fano surface being irreducible
(or see \cite[1.17]{murre}). 

If configuration (c) holds, then as mentioned in Proposition \ref{propseppoints}, the four lines lie in the same
ruling of a smooth quadric $Q\cong\PP^1\times\PP^1$. Any line in the other ruling of the quadric will thus meet all four
lines and so meets the cubic $X$ in at least four points, which implies it must be contained in $X$. This implies the
whole quadric is contained in $X$. Therefore the hyperplane containing the four lines intersects
$X$ in this quadric and a residual 2-plane which as in the cases above is a contradiction. 

We will next prove that configuration (d) of Proposition \ref{propseppoints} does not happen for the general
fibre of the $g^1_4$ of the general member $\cC_t$ of $\cC/T\to S$ a covering family of tetragonal curves for a general
cubic threefold $X$. To this aim, assume $[\ell_1],\ldots,[\ell_4]\in S$ a general fibre of the $g^1_4$ of $\cC_t$ so
that the $\ell_i$ are in configuration (d). Denote by $\Pi=\langle\ell_1,\ell_2\rangle,
\Pi'=\langle\ell_3,\ell_4\rangle$ respectively where $\Pi \cap \Pi' = PP'$. Note that since $t$ is general, we may
assume that $P,P'$ are not Eckardt points, of which there are only finitely many. Since $\Pi$ contains two lines in $X$
meeting at $P$ we have that $\Pi \subset T_PX$ (similarly $\Pi' \subset T_{P'}X$). But then the line $PP'$ is contained
in $X$: indeed, $PP'\subset T_PX\cap T_{P'}X$ and so $PP'$ meets $X$ with multiplicity at least two at each
point and therefore $\ell_5:=PP'\subset X$ (see Figure \ref{sketch}). Note that Lemma \ref{triplepointisspecial} implies
that $P,P'$ lie on the Hessian of $X$. We thus obtain a map $\PP^1\dashrightarrow S$, by sending a point on $\PP^1$ (the
base of the $g^1_4$) to the corresponding $[\ell_5]$ of the fibre of the $g^1_4$ of $\cC_t$, which extends to a morphism.
Since $S$ does not contain any rational curves from Fact \ref{factnortlcurves}, this morphism is constant. Therefore to
the general member $\cC_t$ we associate a fixed line $\ell_{5,t}$. Since there are at most six lines through all but
finitely many points of $X$ (by Lemma \ref{triplepointisspecial}) as we vary along the fibres of the $g^1_4$ on $\cC_t$,
the corresponding points $P,P'$ vary continuously and thus they cover the line $\ell_{5,t}$. This implies that the line
$\ell_{5,t}$ is contained in the intersection $X\cap H$ with the Hessian of $X$. Note also that the image of $\cC_t$ is
a component of the incidence divisor $D_{\ell_{5,t}}$. As we vary $t\in T$ the induced lines $\ell_{5,t}$ vary
continuously, otherwise the image of $\cC$ in $S$ would be contained in a divisor. In particular $H\cap X$ is uniruled
(by lines), which does not happen for the general $X$ from Corollary \ref{generalhessiannotuniruled}.

Finally, to show that configuration (e) only occurs for the family of incidence divisors of second type lines
for any smooth cubic threefold $X$ one proceeds as follows. Assume we have a covering family of tetragonal curves. To
spell out the assumption, we know that the general fibre of the $g^1_4$ of the general member $\cC_t$ will consist of
four points $[\ell_1],\ldots,[\ell_4]\in S$ so that the $\ell_i$ pass through the same point and no three of which are
coplanar. For a fixed $t\in T$, like with three lines through a point in Lemma \ref{3.1}, we obtain a morphism
$f_t:\bP^1\to X$, with image $R=R_t$, from the base of the $g^1_4$, so that through a general point in the image, four
of the lines correspond to the points in the fibres of the $g^1_4$. Since by Lemma \ref{triplepointisspecial}, there are
six lines through all but finitely many points in $X$, we obtain a residual two $L_r,M_r$ through every point $r\in R$. 

If one of the two, say $L_r$, is constant as we vary $r$, then we can define a map $R\dashrightarrow S$ sending $r$ to
$[M_r]$, but this necessarily has constant image by Fact \ref{factnortlcurves} so $M_r$ is also constant. If $L_r$ and
$M_r$ are both constant as we vary $r$, they have to both coincide with $R$ as they each intersect $R$ at every point
and $R$ is irreducible. Hence $L:=M_r=L_r$ is of second type from Fact \ref{factlines}, and the covering family is by
construction the family of Lemma \ref{gonalitydell}.

We may now assume that $L_r$ and $M_r$ both move. This traces out a curve $H_t \subset S$ defined as follows: take the
curve $\tilde{H}_t $ in $S \times S$ consisting of the pairs of points $([L_r],[M_r])$ and $([M_r],[L_r])$, $r\in R$. If
$p:S\times S \to S$ is the first projection then define $H_t$ to be the (reduced) image of this map. We next show that
$\tilde{H}_t$ is irreducible. Indeed, otherwise it would consist of two rational components (since it is hyperelliptic)
both of which get contracted to fixed points in $S$ (from Fact \ref{factnortlcurves}) which is a contradiction to the
assumption that the two lines move. Thus we may assume that $\tilde{H}_t$ is irreducible and let $\tilde{h}: \tilde{H}_t
\to R_t$ be the natural hyperelliptic map. Then $H_t$ is hyperelliptic too as the image under a finite map of a
hyperelliptic curve. The induced hyperelliptic map $h: H_t \to \bP^1$ commutes with $\tilde{h}$, i.e.\ there is a finite
map $\phi: \bP^1 \to \bP^1$ with $\phi \circ \tilde{h}= h \circ p$. We actually claim that the degree of the map $p:
\tilde{H}_t\to H_t$ is one. Otherwise there exist pairs of lines $([L],[M])$ and $([L],[M'])$ with $L,M,M'$ all
distinct. Since $\tilde{h}([L],[M])=\tilde{h}([M],[L])$ we get $\phi\tilde{h}([L],[M])=\phi\tilde{h}([M],[L])$ and so
$hp([L],[M])=hp([M],[L])$, i.e.\ $h([L])=h ([M])$; similarly, $h([L])=h([M'])$ and therefore the three distinct points
$[L],[M],[M']$ of $H_t$ are in the same fibre of the hyperelliptic map which is a contradiction. 

Now we let $t\in T$ vary and we may assume (after a finite base change) from Lemma \ref{dgonalcovfamily} that we have a
family $\mathcal{H}\to T$ of smooth hyperelliptic curves so that $\mathcal{H}_t$ has image $H_t$ in $S$. From Lemma
\ref{bva1} since $S$ is not covered by hyperelliptic curves, the image of $\mathcal{H}\to S$ is one dimensional. Hence
there has to be an open subset $U\subset T$ so that every fibre of $\mathcal{H}|_U\to U$ dominates a fixed irreducible
component $H$ of the above image, which necessarily has to be hyperelliptic as the image of hyperelliptic curves. In
fact one sees that $H$ admits infinitely many $g^1_2$: this follows since the data of $R_t$ along with two lines through
a point $r\in R_t$ is equivalent to that of a $g^1_2$ on $H$, and since $R_t$ moves in a family in $X$, we obtain
different pairs of lines. Now since $H$ has infinitely many $g^1_2$, so does its resolution which implies that $H$ has
geometric genus one or zero. As we have seen before, the latter cannot happen as $S$ does not contain any rational
curves from Fact \ref{factnortlcurves}. On the other hand if $E\to S$ the induced morphism from an elliptic curve
resolving $H$, it has to be an embedding from Lemma \ref{ellcurveembedded} and so $E=H$ is contained smoothly in $S$.
From Fact \ref{roulleauellcurves} this does not happen for the general cubic threefold.

Moving on to pentagonal covering families, note first of all that the family of incidence divisors to lines of first
type does indeed have a natural $g^1_5$ from Lemma \ref{gonalitydell}. To prove now that this family of curves is
unique, we work as above through the various cases of Propositions \ref{propfivelinesuniquesecant},
\ref{propfivelinestwospanaplane}, \ref{propfivelinesthreespanahyperplane}, showing in each of them that if there exists
a pentagonal covering family with such a configuration, then the curves themselves must be incidence divisors $D_\ell$
with a possibly different $g^1_5$ than the natural one above.

Assume that the general fibre of the general member of a pentagonal covering family of curves corresponds to five lines
in the configuration of Proposition \ref{propfivelinesuniquesecant}. Suppose that the fibres of the general member $C$
are parametrized by $t$ in the base ${\mathbb P}^1$ of the $g^1_5$. Then from Proposition
\ref{propfivelinesuniquesecant}, to the generic $t$ there corresponds a unique 5-secant line $\ell_t$, which has to be
contained in the cubic threefold since it intersects it at five points. In particular we get an induced morphism
$\mathbb{P}^1\to S$ sending $t$ to $\ell_t$ which as before must be constant from Fact \ref{factnortlcurves}. Therefore,
the $\ell_t$'s are all equal to a constant line $\ell$ and by construction the corresponding divisor which is the image
of $C$ in the Fano surface is the incidence divisor $D_\ell$. Note that from Fact \ref{factlines}, since the five lines are
distinct generically, $\ell$ is necessarily of first type.

Assume next that the general fibre is in one of the configurations of Proposition \ref{propfivelinestwospanaplane}. The
case of five coplanar lines cannot occur - this can be ruled out like in cases (a), (b) of Proposition
\ref{propseppoints} above. Suppose now the five lines corresponding to the general fibre of the $g^1_5$ of the general
member $C$ of the covering family pass through a point $A$. Then as before, we would obtain a morphism from the base
$\mathbb{P}^1\to S$ sending $t$ to the sixth line passing through $A$, which would necessarily be constant from Fact
\ref{factnortlcurves}. In other words the sixth line of $X$ passing through $A$ has to be the same line $\ell$ for all
the fibres of the $g^1_5$ and the curve $C$ therefore maps to the divisor $D_\ell$ in $S$. 

Finally assume that we are in one of the configurations of Proposition \ref{propfivelinesthreespanahyperplane}. For case
(1), we argue as we did above for case (c) of Proposition \ref{propseppoints}, namely by observing that necessarily
$Q\subset X$ which leads to a contradiction. For case (2), observe first that the secants are contained in the cubic
threefold $X$. If the five lines admit a unique secant line, we can define a map $\bP^1\to S$ and conclude as usual.
Otherwise if the five lines admit exactly two secant lines then we can construct a hyperelliptic curve in the Fano
surface in a similar way as above where we ruled out case (e) of Proposition \ref{propseppoints}. Again this will have
infinitely many $g^1_2$'s since two lines have finitely many secants contained in $X$, corresponding to the intersection
of their incidence divisors, so the two secants must move as we vary along the base of the covering family of curves. We
conclude by noting that there are no elliptic curves in the Fano surface of a general cubic threefold.

\subsection{The proof of Theorem \ref{degofirr}}

If $X$ a smooth cubic threefold and $S$ its Fano surface, then a general hyperplane $H=\PP^3$ cuts $X$ in a smooth cubic
surface $Y$. There is a natural map $S\dashrightarrow Y$ taking an $s\in S$ and giving the point of intersection of the
line $\ell_s$ with the hyperplane $H$. This is generically of degree six since for a general point of $Y\subset X$ there
are six lines through it which are contained in $X$ from Lemma \ref{triplepointisspecial}. On the other hand, the cubic
surface $Y\subset\PP^3$ is rational. In particular the composition $S\dashrightarrow Y\to\PP^2$ has degree six and so
$\irr(S)\leq 6$.

Now, assume that $\cng(S)=4$. Then there has to be an at least one dimensional family of tetragonal curves through the
general point of $S$ (i.e.\ a two-dimensional variety $T$ and a tetragonal family of curves $\cC\to T$ making the
induced double evaluation morphism $\cC\times_T\cC\to S\times S$ dominant). This however is not possible for $X$ general
since the unique tetragonal covering family of incidence divisors of second type lines moves in a one dimensional family
from Proposition \ref{covgonfour}. Since $\cng(S)\leq\irr(S)$ we obtain $\cng(S)=5\leq \irr(S)\leq6$.

Finally, if $\irr(S)=5$ then there exists a dominant rational map $F:S \dashrightarrow {\mathbb P}^2$ which is generically
finite of degree five. Then through a general point $x\in S$ there passes a 4-dimensional family of curves of gonality
less than or equal to five - indeed, through the point $F(x)$ there passes a 4-dimensional family of conics and the
preimage of this gives the desired family. On the other hand Proposition \ref{covgonfour} proves that the only 5-gonal
covering family of curves is that of incidence divisors to lines, and the base of this family is only two dimensional, a
contradiction.

\bibliographystyle{alpha}

\begin{thebibliography}{bib}

\bibitem[ACGH85]{acgh}
Enrico Arbarello, Maurizio Cornalba, Phillip A.~Griffiths, and Joseph D.~Harris.
\newblock {\em Geometry of algebraic curves. {V}ol. {I}}, volume 267 of {\em
  Grundlehren der Mathematischen Wissenschaften [Fundamental Principles of
  Mathematical Sciences]}.
\newblock Springer-Verlag, New York, 1985.

\bibitem[AK77]{ak}
Allen B.~Altman and Steven L.~Kleiman.
\newblock {Foundations of the theory of Fano schemes}.
\newblock {\em Compositio Mathematica}, 34(1):3--47, 1977.

\bibitem[AK03]{kollararaujo}
Carolina Araujo and J\'anos Koll\'ar.
\newblock Rational curves on varieties.
\newblock In {\em Higher dimensional varieties and rational points ({B}udapest,
  2001)}, volume~12 of {\em Bolyai Soc. Math. Stud.}, pages 13--68. Springer,
  Berlin, 2003.

\bibitem[AR96]{adler}
Allan Adler and Sundararaman Ramanan.
\newblock {\em Moduli of abelian varieties}, volume 1644 of {\em Lecture Notes
  in Mathematics}.
\newblock Springer-Verlag, Berlin, 1996.

\bibitem[Bas12]{bastianelli}
Francesco Bastianelli.
\newblock {On symmetric products of curves}.
\newblock {\em Transactions of the American Mathematical Society},
  364(5):2493--2519, 2012.

\bibitem[BCFS17]{bcfs}
Francesco Bastianelli, Ciro Ciliberto, Flaminio Flamini, and Paola Supino.
\newblock {Gonality of curves on general hypersufaces}.
\newblock {\em arXiv.org}, July 2017.

\bibitem[BD85]{bd}
Arnaud Beauville and Ron Donagi.
\newblock {La vari\'et\'e des droites d'une hypersurface cubique de dimension
  4}.
\newblock {\em Comptes Rendus des S\'eances de l'Acad\'emie des Sciences.
  S\'erie I. Math\'ematique}, 301(14):703--706, 1985.

\bibitem[BDPELU17]{bdelu}
Francesco Bastianelli, Pietro De~Poi, Lawrence Ein, Robert Lazarsfeld, and
  Brooke Ullery.
\newblock Measures of irrationality for hypersurfaces of large degree.
\newblock {\em Compos. Math.}, 153(11):2368--2393, 2017.

\bibitem[Bea77]{beauville77}
Arnaud Beauville.
\newblock {Vari\'et\'es de Prym et jacobiennes interm\'ediaires}.
\newblock {\em Annales Scientifiques de l'\'Ecole Normale Sup\'erieure.
  Quatri\`eme S\'erie}, 10(3):309--391, 1977.

\bibitem[BSD67]{bsd}
Enrico Bombieri and H.~Peter~F.~Swinnerton-Dyer.
\newblock {On the local zeta function of a cubic threefold}.
\newblock {\em Ann. Scuola Norm. Sup. Pisa (3)}, 21:1--29, 1967.

\bibitem[CG72]{cg}
C.~Herbert Clemens and Phillip~A. Griffiths.
\newblock {The intermediate Jacobian of the cubic threefold}.
\newblock {\em Annals of Mathematics. Second Series}, 95(2):281--356, 1972.

\bibitem[CLMTiB17]{clb}
Abel Castorena, Alberto L\'opez~Martin, and Montserrat Teixidor~i Bigas.
\newblock Invariants of the {B}rill--{N}oether curve.
\newblock {\em Adv. Geom.}, 17(1):39--52, 2017.

\bibitem[dFF13]{dff}
Tommaso de~Fernex and Davide Fusi.
\newblock Rationality in families of threefolds.
\newblock {\em Rend. Circ. Mat. Palermo (2)}, 62(1):127--135, 2013.

\bibitem[Dol12]{dolgachev}
Igor~V. Dolgachev.
\newblock {\em Classical algebraic geometry}.
\newblock Cambridge University Press, Cambridge, 2012.
\newblock A modern view.

\bibitem[FP97]{fp}
William Fulton and Rahul Pandharipande.
\newblock Notes on stable maps and quantum cohomology.
\newblock In {\em Algebraic geometry---{S}anta {C}ruz 1995}, volume~62 of {\em
  Proc. Sympos. Pure Math.}, pages 45--96. Amer. Math. Soc., Providence, RI,
  1997.

\bibitem[GH94]{griffithsharris}
Phillip Griffiths and Joseph Harris.
\newblock {\em Principles of algebraic geometry}.
\newblock Wiley Classics Library. John Wiley \& Sons, Inc., New York, 1994.
\newblock Reprint of the 1978 original.

\bibitem[Har92]{harris}
Joseph Harris.
\newblock {\em Algebraic geometry. A first course}, volume 133 of {\em Graduate Texts in
  Mathematics}.
\newblock Springer-Verlag, New York, 1992.


\bibitem[HM98]{harrismorrison}
Joseph Harris and Ian Morrison.
\newblock {\em Moduli of curves}, volume 187 of {\em Graduate Texts in
  Mathematics}.
\newblock Springer-Verlag, New York, 1998.

\bibitem[HPT16]{htp}
Brendan Hassett, Alena Pirutka, and Yuri Tschinkel.
\newblock {Stable rationality of quadric surface bundles over surfaces}.
\newblock {\em arXiv.org}, March 2016.

\bibitem[Kol96]{kollar}
J\'anos Koll\'ar.
\newblock {\em Rational curves on algebraic varieties}, volume~32 of {\em
  Ergebnisse der Mathematik und ihrer Grenzgebiete. 3. Folge. A Series of
  Modern Surveys in Mathematics [Results in Mathematics and Related Areas. 3rd
  Series. A Series of Modern Surveys in Mathematics]}.
\newblock Springer-Verlag, Berlin, 1996.

\bibitem[KT17]{kt}
Maxim Kontsevich and Yuri Tschinkel.
\newblock {Specialization of birational types}.
\newblock {\em arXiv.org}, August 2017.

\bibitem[LP95]{lopezpirola}
Angelo~Felice Lopez and Gian~Pietro Pirola.
\newblock {On the curves through a general point of a smooth surface in
  P$^{3}$}.
\newblock {\em Mathematische Zeitschrift}, 219(1):93--106, 1995.

\bibitem[Mur72]{murre}
Jacob P. Murre.
\newblock {Algebraic equivalence modulo rational equivalence on a cubic
  threefold}.
\newblock {\em Compositio Mathematica}, 25:161--206, 1972.

\bibitem[NS17]{ns}
Johannes Nicaise and Evgeny Shinder.
\newblock {The motivic nearby fiber and degeneration of stable rationality}.
\newblock {\em arXiv.org}, August 2017.

\bibitem[Rou09a]{roulleauelliptic}
Xavier Roulleau.
\newblock Elliptic curve configurations on {F}ano surfaces.
\newblock {\em Manuscripta Math.}, 129(3):381--399, 2009.

\bibitem[Rou09b]{roulleauklein}
Xavier Roulleau.
\newblock {The Fano surface of the Klein cubic threefold}.
\newblock {\em Journal of Mathematics of Kyoto University}, 49(1):113--129,
  2009.

\bibitem[Rou11]{roulleau1230}
Xavier Roulleau.
\newblock {Fano surfaces with 12 or 30 elliptic curves}.
\newblock {\em Michigan Mathematical Journal}, 60(2):313--329, 2011.

\bibitem[Tju71]{tjurin}
Andrei N.~Tjurin.
\newblock The geometry of the {F}ano surface of a nonsingular cubic {$F\subset
  P^{4}$}, and {T}orelli's theorems for {F}ano surfaces and cubics.
\newblock {\em Izv. Akad. Nauk SSSR Ser. Mat.}, 35:498--529, 1971.

\bibitem[Voi03]{voisinnotes}
Claire Voisin.
\newblock On some problems of {K}obayashi and {L}ang; algebraic approaches.
\newblock In {\em Current developments in mathematics, 2003}, pages 53--125.
  Int. Press, Somerville, MA, 2003.

\bibitem[Voi14]{voisinbook}
Claire Voisin.
\newblock {\em Chow rings, decomposition of the diagonal, and the topology of
  families}, volume 187 of {\em Annals of Mathematics Studies}.
\newblock Princeton University Press, Princeton, NJ, 2014.

\bibitem[Yos94]{yoshihara}
Hisao Yoshihara.
\newblock {Degree of irrationality of an algebraic surface}.
\newblock {\em Journal of Algebra}, 167(3):634--640, 1994.

\end{thebibliography}

\newcommand{\etalchar}[1]{$^{#1}$}

\end{document}